\documentclass[12pt]{amsart}
\usepackage{amssymb}
\usepackage{amsfonts}
\usepackage{epsfig}
\usepackage{fancyhdr}

\makeatletter
\@namedef{subjclassname@2010}{%
  \textup{2010} Mathematics Subject Classification}
\makeatother

\def\Div{{\rm div\,}}
\def\rot{{\rm rot}\,}

\def\R{\rm R}

\def\N{\mathcal{N}}

\newcommand{\al}{\alpha}

\def\Ga{\Gamma}

\newcommand{\eps}{\varepsilon}
\newcommand{\va}{\varphi}

\def\si{\sigma}
\def\De{\Delta}

\def\nb{\nabla}
\def\eps{\varepsilon}
\def\Om{\Omega}
\def\om{\omega}
\def\pa{\partial}

\def\Or{\int_{\R^3}}

\def\ra{\rightarrow}
\def\nn{\nonumber}

\newcommand{\benn}{\begin{eqnarray*}}
\newcommand{\eenn}{\end{eqnarray*}}
\newcommand{\ben}{\begin{eqnarray}}
\newcommand{\een}{\end{eqnarray}}
\newcommand{\bal}{\begin{aligned}}
\newcommand{\eal}{\end{aligned}}

\theoremstyle{plain}

\newtheorem{lemma}{Lemma}
\newtheorem{theorem}{Theorem}

\newtheorem{remark}[lemma]{Remark}
\renewcommand{\thesection}{\arabic{section}}

\newtheorem{corollary}[lemma]{Corollary}

\newcommand{\appsection}[1]{\let\oldthesection\thesection
  \renewcommand{\thesection}{Appendix \oldthesection}
  \section{#1}\let\thesection\oldthesection}

\def\\{\hfil\break}

\def\N{{\Bbb N}}
\def\R{{\Bbb R}}

\title[NS regularity criteria]{On some regularity criteria for axisymmetric Navier-Stokes equations.}

\author[J.Renc{\L}awowicz \and W. M. Zaj\c{a}czkowski]{Joanna Renc{\L}awowicz
\and Wojciech M. Zaj\c{a}czkowski}

\keywords{Navier-Stokes equation,  regularity criteria,
 regular solutions}

\thanks{Institute of Mathematics, Polish Academy of
Sciences, \'{S}nia\-dec\-kich 8, 00-656 Warsaw, Poland, e-mail:
jr@impan.gov.pl}
\thanks{Institute of Mathematics, Polish Academy of
Sciences, \'{S}niadeckich 8, 00-656 Warsaw, Poland, and Institute of Mathematics and Cryptology,
Military University of Technology, Kaliskiego 2, 00-908 Warsaw,
Poland, e-mail:
wz@impan.gov.pl}
\thanks{\today}

\begin{document}


\begin{abstract}
We point out some criteria that imply regularity of axisymmetric solutions to Navier-Stokes
equations. We show that boundedness of $\|{v_{r}}/{\sqrt{r^3}}\|_{L_2(\R^3\times (0,T))}$ as well as
boundedness of $\|{\omega_{\va}}/{\sqrt{r}}\|_{L_2(\R^3\times (0,T))},$ where $v_r$ is the radial component of velocity
and $\omega_{\va}$ is the angular component of vorticity, imply regularity of weak solutions.
\end{abstract}


\pagestyle{myheadings}
\thispagestyle{plain}
\markboth{J. Renc{\l}awowicz \& W. M. Zaj\c{a}czkowski}{NS regularity criteria}

\maketitle

\section{Introduction}

We consider the Cauchy problem to the three-dimensional axisymmetric Navier-Stokes equations:

\ben \bal \label{NS} %
&v_{t}+v\cdot\nabla v- \nu\De v + \nb p =0 \quad & (x,t) \in
 \R^3 \times{\R_+}, \\
&\Div v=0,\quad & \\
 &v\big|_{t=0}=v(0),\quad & \eal \een \noindent
where  $x= (x_1, x_2, x_3)$, $v$ is the velocity of the fluid motion with
$$v(x,t)=(v_1(x,t),v_2(x,t),v_3(x,t))\in\R^3,$$
$p=p(x,t)\in\R^1$ denotes the pressure,
and $v_0$ is given initial velocity field.

The first papers concerning regularity of axially symmetric solutions to the Navier-Stokes equations were independently proved by Ladyzhenskaya \cite{L} and Yudovich-Ukhovskij \cite{YU} in 1968. In these papers axisymmetric solutions without swirl were considered. In the period 1999-2002 arised many papers concerning sufficient conditions on regularity of axisymmetric solutions (\cite{CL}, \cite{NP1}, \cite{NP2}, \cite{LMNP}). Especially, conditions on one coordinate of velocity were considered. Recently there are many papers dealing with new sufficient conditions (see references of Lei and Zhang \cite{LZ}).

Our aim is to derive some criteria guaranteeing regularity of solutions to the
axisymmetric Navier-Stokes equations. By the regular solutions we mean smooth weak solutions
obtained by the standard increasing regularity technique for smooth initial data.There is a
lot of criteria for regularity of axisymmetric solutions (see \cite{CFZ}, \cite{NP1},
\cite{KP}, \cite{KPZ}, \cite{CT}, \cite{LZ}, \cite{Z} and the literature cited in these papers). In Section~2 we recall only
such criteria  that are useful for our analysis.

Since we are restricted to the axisymmetric solutions we introduce the cylindrical coordinates $(r,\va,z)$ by the relations
\benn
x_1=r\cos\varphi,\quad x_2=r\sin\varphi,\quad x_3=z,
\eenn
and corresponding unit vectors:
\benn
\bar e_r=(\cos\varphi,\sin\varphi,0),\quad
\bar e_\varphi=(-\sin\varphi,\cos\varphi,0),\quad \bar e_z=(0,0,1).
\eenn
Then the cylindrical components of velocity and vorticity ($\omega=\rot v$) for axisymmetric solutions (therefore, solutions independent of $\va$) are represented as
\benn
v = v_r (r,z,t) \bar e_r + v_\va(r,z,t) \bar e_\va +v_z(r,z,t) \bar e_z
\eenn
and
\benn
\omega = \omega_r (r,z,t) \bar e_r + \omega_\va(r,z,t) \bar e_\va +\omega_z(r,z,t) \bar e_z \\ =
- v_{\va,z} \bar e_r + (v_{r,z} - v_{z,r})\bar e_\va + (v_{\va,r}+ \frac{v_{\va}}{r})\bar e_z,
\eenn
where $v_r, v_\va, v_z$ are radial, angular and axial components of velocity.

The axisymmetric motion can be described by the three quantities: $v_{\va}, \omega_{\va}$ and the stream potential $\psi$ which are solutions to the following equations:
\ben \label{1.6-1.9} \bal
v_{\va,t} + v\cdot \nb v_{\va} - \nu (\De - \frac{1}{r^2}) v_{\va} + \frac{v_r}{r} v_{\va} = 0, \\ v_{\va}|_{t=0}= v_{\va}(0),\\
\omega_{\va,t} + v\cdot \nb \omega_{\va} - \nu (\De - \frac{1}{r^2}) \omega_{\va} - \frac{v_r}{r} \omega_{\va} - \frac{2}{r}v_{\va}v_{\va,z} = 0, \\ \omega_{\va}|_{t=0}= \omega_{\va}(0),\\
- (\De\psi - \frac{1}{r^2}\psi)= \omega_{\va},
\eal \een
where $v\cdot \nb= v_r\partial_r + v_z\partial_z, \De=\pa_r^2 + \pa_z^2 + \frac{1}{r}\pa_r $ and
\benn v_r = -\psi_{,z}, \ v_z = \frac{1}{r} (r\psi)_{,r}. \eenn
It is very convenient to introduce quantities $u_1, \omega_1, \psi_1$ by the relations
\benn
v_{\va} = r u_1, \ \omega_{\va} = r \omega_1,\  \psi = r\psi_1
\eenn
Then equations (\ref{1.6-1.9}) simplify to
\ben \label{1.11a} \bal
u_{1,t} + v\cdot \nb u_{1} - \nu \left(\De u_1 + \frac{2}{r} u_{1,r}\right)= 2 u_1 \psi_{1,z}, \\ u_{1}|_{t=0}= v_{1}(0), \eal\een
\ben \label{1.7} \bal
\omega_{1,t} + v\cdot \nb \omega_{1} - \nu \left(\De \omega_1 + \frac{2}{r} \omega_{1,r}\right) = 2u_1 u_{1,z}, \\ \omega_1|_{t=0}= \omega_1(0),\eal \een
\ben \label{1.13}
- \left(\De\psi_1 + \frac{2}{r}\psi_{1,r}\right)= \omega_1,\een
 and
\ben \label{1.14}
v_r = -r \psi_{1,z}, \quad v_z = \frac{1}{r} (r^2\psi_1)_{,r}.
 \een

We prove the following regularity criteria (see (\ref{1.10}), (\ref{1.11})) which are scaling invariant:
\begin{theorem} \label {th1}
1. Let $(v,p)$ be an axisymmetric solution to the Navier-Stokes equations (\ref{NS}) with the axisymmetric initial data and $\Div v(0) = 0$.

2. Assume that $\frac{v^2_{\va}(0)}{r}, \frac{\omega_{\va}(0)}{r}, \frac{\omega_r(0)}{r},
\om_z(0)$ belong to $L_2(R^3)$, and with $u=rv_{\va}, u(0)\in L_{\infty}(\R^3)\bigcap L_s(\R^3), s\ge 3.$

3. Assume that there exists constant $c_1$ such that
\ben \label{1.10}
\int_0^T dt \Or \frac{v_r^2}{r^3}dx \le c_1 < \infty
\een
then $v \in L_{\infty}(0,T;H^1(\R^3_{r_0})),$ where $\R^3_{r_0}= \{x\in \R^3, r<r_0 \}$ and $r_0>0$ is given.
Assume additionally that $v(0)\in B^{2-2/r}_{\si,r}(\R^3_{r_0})$ -Besov space. Then $v\in W^{2,1}_{\si,r}(\R^3_{r_0}\times(0,T)).$
\end{theorem}

\begin{remark}
For $\si> 3, r=2$ we have that $v\in L_{\infty}(\R^3_{r_0} \times(0,T))$
so in view of \cite{CKN} there is no singular points. In $\bar{\R}^3_{r_0}= \{x\in \R^3, r> r_0 \}$ the axisymmetric problem (\ref{NS}) is two-dimensional so local regularity of $v$ is evident.
\end{remark}

\begin{theorem} \label{th2}
Let the assumptions~1,2 of Theorem~1 hold. If
\ben \label{1.11}
\int_0^T dt \Or \frac{\omega_{\va}^2}{r}dx \le c_2 < \infty
\een
then there exists a constant $c_3$ such that
\ben
\int_0^T dt \Or \frac{v_r^2}{r^3}dx \le c_3\int_0^T dt \Or \frac{\omega_{\va}^2}{r} dx\le c_3c_2
\een
\end{theorem}
\begin{remark}
For $r < \infty$, the assumption is fulfilled by Lemma~\ref{le-2.6} and Lemma~\ref{le-2.8} because for $r\le c_0$ holds $$\Or \frac{\omega_{\va}^2}{r} dx \le c_0 \Or \frac{\omega_{\va}^2}{r^2} dx.$$
\end{remark}

\section{Notation and auxiliary results}
\setcounter{equation}{0}
\setcounter{lemma}{0}

By $L_p(\R^N), p\in [1,\infty],$ we denote the Lebesgue space of integrable functions.
By $L_{p,q}(\R^3 \times (0,T))$ we denote the anisotropic Lebesgue space with the following finite norm
\benn
\|u\|_{L_{p,q}(\R^3 \times (0,T))}= \left( \int_0^T \int_{\R^3}(|u(x,t)|^p dx)^{q/p}dt \right)^{1/q},
\eenn
where $p,q \in [1,\infty].$

We define Sobolev spaces $W_p^{2,1}(\R^3 \times (0,T))$ and $W^{2-2/p}_p(\R^3 \times (0,T))$ by
\benn
\|u\|_{W_p^{2,1}(\R^3 \times (0,T))}  =  \left( \int_0^T \int_{\R^3}(|\nb_x^2 u|^p + |u_t|^p +|u|^p)dx dt \right)^{1/p} <  \infty, & &  \\
\|u\|_{W_p^{2-2/p}(\R^3)} = & & \\ \left(\sum_{i\le [2-2/p]} \int_{\R^3}|\nb^i u|^p dx +
\int_{\R^3}\int_{\R^3} \frac{|\nb_x^{[2-2/p]}u(x) - \nb_y^{[2-2/p]}u(y)|^p}{|x-y|^{3+p(2-2/p
- [2-2/p])}}dx dy\right)^{1/p} < \infty, & &
\eenn where $[l]$ is the integer part of $l$.

By $H^s(\R^3), s\in \N_0 = \N \cup \{0\}$ we denote the Sobolev space $W^s_2(\R^3).$

\begin{lemma} \label{le-2.1}
There exists a weak solution to problem (\ref{NS}) such that
$v \in L_{\infty}(0,T;L_2(\R^3))\cap L_2(0,T;H^1(\R^3))$
and the following estimate holds
\ben \label{2.1}
\int_{\R^3} |v(t)|^2 dx + \nu \int_0^{t} dt' \int_{\R^3}|\nb v|^2 dx\le c \int_{\R^3}|v(0)|^2 dx.
\een
In the case of axisymmetric solutions the energy inequality (\ref{2.1}) takes the form
\ben \label{2.2} \bal
\qquad \int_{\R^3} |v(t)|^2 dx + \nu \int_0^{t} dt' \int_{\R^3}\left(|\nb v|^2+\left|\frac{v_r}{r}\right|^2 +\left|\frac{v_{\va}}{r}\right|^2 \right)dx \\ \le c \int_{\R^3}|v(0)|^2 dx.
\eal \een
\end{lemma}
\begin{proof}
Equations $(1.1)_{1,2}$ for the axially symmetric
solutions assume the form
\ben \label{2.3}
v_{r,t}+v\cdot\nabla v_r-{v_\varphi^2\over r}-\nu\Delta v_r+\nu
\frac{v_r}{r^2}=-p_{,r},
\een
\ben \label{2.4}
v_{\varphi,t}+v\cdot\nabla v_\varphi+{v_r\over r}v_\varphi-\nu\Delta v_\varphi+
\nu \frac{v_\varphi}{r^2}=0
\een
\ben \label{2.5}
v_{z,t}+v\cdot\nabla v_z-\nu\Delta v_z=-p_{,z},
\een
\ben \label{2.6}
v_{r,r}+v_{z,z}=-{v_r\over r},
\een
where $v\cdot\nabla=v_r\partial_r+v_z\partial_z$,
$\Delta u=\frac{1}{r}(ru_{,r})_{,r}+u_{,zz}$.

\noindent

Let $I= \{ (\va, r, z): r=0\}$ denote the axis of symmetry. Define the space $X$ as the closure of $C^{\infty}_0 (\R^3\setminus I)$ in the $X$ norm. Then, we are looking for a priori estimate for functions $v\in X$.

Multiplying (\ref{2.4}) by $v_\varphi$ and integrating over $\R^3$ yields
$$
\frac{1}{2}\frac{d}{dt}\Or v_\varphi^2dx+\Or
\frac{v_r}{r}v_\varphi^2dx+\nu\Or(v_{\varphi,r}^2+
v_{\varphi,z}^2)dx
+\nu\Or{v_\varphi^2\over r^2}dx=0.
$$
Multiplying (\ref{2.3}) by $v_r$, integrating over $\R^3$ implies
$$
\frac{1}{2}\frac{d}{dt}\Or v_r^2dx-\Or
\frac{v_\varphi^2}{r}v_rdx+\nu\Or(v_{r,r}^2+v_{r,z}^2)dx+
\nu\Or \frac{v_r^2}{r^2}dx =-\Or p_{,r}v_rdx.
$$
Multiplying (\ref{2.5}) by $v_z$ and integrating over $\R^3$ we obtain
$$
\frac{1}{2}\frac{d}{dt}\Or v_z^2dx+\nu\Or(v_{z,r}^2+v_{z,z}^2)
dx=-\Or p_{,z}v_zdx.
$$
Adding the above equations and using (\ref{2.6})
we obtain
\ben \label{2.22}
\frac{1}{2}\frac{d}{dt}\Or(v_r^2+v_\varphi^2+v_z^2)dx+
\\ \nu \Or(v_{r,r}^2+v_{r,z}^2+v_{\varphi,r}^2+v_{\varphi,z}^2+v_{z,r}^2+
v_{z,z}^2)dx
+\nu\Or\bigg(\frac{v_r^2}{r^2}+\frac{v_\varphi^2}{r^2}\bigg)dx
=0.\nn
 \een
Integrating (\ref{2.22}) with respect to time from $0$ to $t$,
$t\le T,$  yields
\benn
\nu \|v(t)\|_{L_2(\R^3)}^2+\nu\intop_{0}^t\Or(|v_{,r}|^2+|v_{,z}|^2)
dxdt'+\nu\intop_{0}^t\Or\bigg(\frac{v_r^2}{r^2}+
\frac{v_z^2}{r^2}\bigg)dxdt'\\
\le \frac{1}{2} \|v(0)\|_{L_2(\R^3)}^2.
\eenn
 This ends the proof.

\end{proof}

To derive energy estimates in the proof of Lemma~\ref{le-2.1} we use the ideas from the proof of Theorem~3.1 from \cite{T}, Ch.3.
The notion of a suitable weak solution was introduced by L. Caffarelli, R. Kohn and L. Nirenberg in the famous paper \cite{CKN}. Our aim is to show that either (\ref{1.10}) or (\ref{1.11}) implies that a suitable weak solution to problem (\ref{NS}) does not contain singular points.This means that $(v,p)$ is a regular solution to (\ref{NS}). In other words it means that if $v(0)\in W_p^{2-2/p}(\R^3)$ then $v\in W^{2,1}_p(\R^3 \times R_+)$ for any $p\in (1,\infty)$. Hence for $p>\frac{5}{2}$ we have that $v\in L_{\infty}(\R^3 \times R_+)$ so it is also bounded locally. Therefore $v$ has no singular points (see \cite{CKN}). To show this we use results of J. Neustupa, M. Pokorny and O. Kreml (see \cite{NP1}, \cite{NP2}, \cite{KP}). To clarify presentation we recall the results.

From \cite{NP2} it follows that
\begin{lemma} \cite{NP2} \label{le-np2}
Let $v$ be an axisymmetric suitable weak solution to problem (\ref{NS}). Suppose that there exists a subdomain $D \subset \R^3\times \R_+$ such that the angular component $v_{\va}$ of $v$ belongs to $L_{s,r}(D)$ where

1. either $s\in [6,\infty], r\in[20/7, \infty]$ and $2/r+3/s\le 7/10.$

2. or $s\in [24/5,6], r\in [10,\infty]$ and $2/r+3/s\le 1-9/(5s).$

Then $v$ has no singular points in $D.$

\end{lemma}
\begin{lemma}[KP] \label{le-KP}
Let $v$ be an axisymmetric suitable weak solution to problem (\ref{NS}). Suppose that there exists a subdomain $D \subset \R^3\times \R_+$ such that the angular component $v_{\va}$ of $v$ belongs to $L_{s,r}(D)$ where
$$ s\in \left(\frac{24}{7}, 4 \right], r\in \left(\frac{8s}{7s-24},\infty\right], \frac{3}{s}+\frac{2}{r} < \frac{7}{4} - \frac{3}{s}.$$
Then $v$ has no singular points in $D.$
\end{lemma}


By swirl we denote
\benn u=rv_{\va} \eenn
From $(\ref{1.6-1.9})_1$ it follows that $u$ satisfies the equation
\benn
u_{,t}+ v\cdot \nb u -\nu \De u+ \frac{2\nu}{r} u_{,r} = 0
\eenn
\begin{lemma}(see \cite{CL}) \label{le-2.4}
Let $u(0)= r v_{\va}(0) \in L_{\infty}(\R^3).$ Then
\ben \label{2.8}
\|u\|_{L_{\infty}(0,T;L_{\infty}(\R^3))} \le \|u(0)\|_{L_{\infty}(\R^3)} \quad {\rm for \  any}\  T.
\een
\end{lemma}

\begin{remark}
From (\ref{2.2}) and (\ref{2.8}) we have
\ben \label{2.9} \bal \intop_{0}^t \Or |v_{\va}|^4 dx dt = \intop_{0}^t \Or r^2 v_{\va}^2{v_\varphi^2\over r^2} dx dt \\ \le \|r v_{\va}\|_{L_{\infty}(\R^3 \times (0,T))}^2\intop_{0}^t \Or {v_\varphi^2\over r^2} dx dt \le c\|u(0)\|^2_{L_{\infty}(\R^3)} \|v(0)\|^2_{L_2(\R^3)} \eal \een
\end{remark}

We recall also Lemma~3.1 from \cite{CFZ}
\begin{lemma}(see \cite{CFZ}) \label{le-2.6}
 Assume that $(v,p)$ is regular axisymmetric solution to the Navier-Stokes equations (\ref{NS}). Assume that $\frac{v_{\va}(0)}{\sqrt{r}}\in L_4(\R^3), \tilde{\nb}\frac{v_r}{r} \in L_{4/3}(0,T; L_2(\R^3))$ when $\tilde{\nb}=(\pa_r, \pa_z).$ Then the following estimate holds
\ben \label{2.10}
\frac{1}{2} \left\|\frac{v_{\va}^2(t)}{r}\right\|_{L_2(\R^3)}^2 + \frac{1}{4} \int_0^t \left\|\tilde{\nb}\frac{v_{\va}^2}{r}\right\|_{L_2(\R^3)}^2 dt'  + \frac{3}{4} \int_0^t \left\|\frac{v_{\va}}{r}\right\|_{L_4(\R^3)}^4 dt' \nn \\ \le \frac{1}{4}\exp[c\int_0^t \left\|\tilde{\nb}\frac{v_r}{r}\right\|_{L_2(\R^3)}^{4/3} dt' ] \left\|\frac{v_{\va}^2(0)}{r}\right\|_{L_2(\R^3)}^2.
\een

\end{lemma}

\begin{proof}
Consider the following problem for $v_{\va}$ in \\ $\R^3_{\eps}= \{x\in \R^3: r>\eps \}, \eps>0,$

\ben \label{2.11} \bal v_{\va,t} + v\cdot\nb v_{\va}-\nu \De v_{\va} + \nu \frac{v_{\va}}{r^2} + \frac{v_r v_{\va}}{r} =0, \\ v_{\va}|_{r=\eps}=0, \quad v_{\va}|_{r\rightarrow \infty} =0  \eal\een
Multiplying (\ref{2.11}) by $\frac{v_{\va}^3}{r^2},$ integrating over $\R^3_{\eps}$ and using boundary conditions yields
\benn
\frac{1}{4} \frac{d}{dt} \left\|\frac{v_{\va}^2}{r}\right\|_{L_2(\R^3_{\eps})}^2 + \frac{3}{4} \nu \left\|\tilde{\nb}\frac{v_{\va}^2}{r}\right\|_{L_2(\R^3_{\eps})}^2 +  \frac{3}{4}\nu
\left\|\frac{v_{\va}}{r}\right\|_{L_4(\R^3_{\eps})}^4  = - \frac{3}{2}\int_{\R^3_{\eps}}\frac{v_r}{r}\frac{v_{\va}^2}{r}\frac{v_{\va}^2}{r} dx \\ \le \frac{3}{2} \left\|\frac{v_r}{r}\right\|_{L_6(\R^3_{\eps})} \left\|\frac{v_{\va}^2}{r}\right\|_{L_2(\R^3_{\eps})} \left\|\frac{v_{\va}^2}{r}\right\|_{L_3(\R^3_{\eps})} \le c \left\|\tilde{\nb}\frac{v_r}{r}\right\|_{L_2(\R^3_{\eps})} \left\|\frac{v_{\va}^2}{r}\right\|^{3/2}_{L_2(\R^3_{\eps})} \left\|\tilde{\nb}\frac{v_{\va}^2}{r}\right\|^{1/2}_{L_2(\R^3_{\eps})} \\
\le \frac{\nu}{2} \left\|\tilde{\nb}\frac{v_{\va}^2}{r}\right\|^2_{L_2(\R^3_{\eps})}
+ c \left\|\frac{v_{\va}^2}{r}\right\|^{2}_{L_2(\R^3_{\eps})}
\left\|\tilde{\nb}\frac{v_{r}}{r}\right\|^{4/3}_{L_2(\R^3_{\eps})}.
\eenn
Simplifying we have
\benn
\frac{1}{4} \frac{d}{dt} \left\|\frac{v_{\va}^2}{r}\right\|_{L_2(\R^3_{\eps})}^2
+ \frac{\nu}{4} \left\|\tilde{\nb}\frac{v_{\va}^2}{r}\right\|_{L_2(\R^3_{\eps})}^2
+  \frac{3}{4}\nu \left\|\frac{v_{\va}}{r}\right\|_{L_4(\R^3_{\eps})}^4
\le c \left\|\frac{v_{\va}^2}{r}\right\|^{2}_{L_2(\R^3_{\eps})}
\left\|\tilde{\nb}\frac{v_{r}}{r}\right\|^{4/3}_{L_2(\R^3_{\eps})}.
\eenn
By the Gronwall lemma we have
\benn \frac{1}{4} \left\|\frac{v_{\va}^2(t)}{r}\right\|_{L_2(\R^3_{\eps})}^2
+ \frac{\nu}{4} \int_0^t \left\|\tilde{\nb}\frac{v_{\va}^2}{r}\right\|_{L_2(\R^3_{\eps})}^2 dt'
+ \frac{3}{4}\nu \int_0^t \left\|\frac{v_{\va}}{r}\right\|_{L_4(\R^3_{\eps})}^4 dt'
\\ \le \frac{1}{4}\exp[c\int_0^t \left\|\tilde{\nb}\frac{v_r}{r}\right\|_{L_2(\R^3_{\eps})}^{4/3} dt' ]
\left\|\frac{v_{\va}^2(0)}{r}\right\|_{L_2(\R^3_{\eps})}^2. \eenn
Passing with $\eps \rightarrow 0$ we derive (\ref{2.10}) and conclude the proof.
\end{proof}
\begin{remark}
Formula (2.4) in \cite{CFZ} has the form
\ben \label{2.11a} \left\|\tilde{\nb}\frac{v_r}{r}\right\|_{L_q(\R^3)} \le c(q)\left\|\frac{\om_{\va}}{r}\right\|_{L_q(\R^3)}, \quad 1<q<\infty. \een
\end{remark}
Consider problem (\ref{1.11a}).
\begin{lemma} \label{le-2.8}
Let the assumptions of Lemma~2.6 be satisfied. Assume additionally that $w_1(0) \in L_2(\R^3), u_1(0) \in L_4(0,t;L_4(\R^3)).$ Then the following estimate holds
\ben \label{2.12} \bal
\frac{1}{2}\|\om_1(t)\|^2_{L_2(\R^3)}+\frac{\nu}{2}\int_0^t\|\nb \om_1(t')\|^2_{L_2(\R^3)} dt'  \\ \le \frac{2}{\nu}\int_0^t\|u_1(t')\|^4_{L_4(\R^3)} dt'+ \frac{1}{2}\|\om_1(0)\|^2_{L_2(\R^3)}, \quad t\le T.
\eal \een
\end{lemma}
\begin{proof}
Consider the problem in $\R^3_{\eps}$
\ben \label{2.13} \bal
\om_{1,t}+ v\cdot \nb \om_1 - \nu \left(\De \om_1+\frac{2}{\nu}\om_{1,r}\right)= 2u_1u_{1,z} \\
\om_1|_{r=\eps}=0, \quad \om_1|_{r\rightarrow \infty} =0  \eal
\een
Multiplying (\ref{2.13}) by $\om_1$, integrating over $\R^3_{\eps}$, using the boundary conditions yields
\benn  \frac{1}{2}\frac{d}{dt}\|\om_1(t)\|^2_{L_2(\R^3_{\eps})}+\frac{\nu}{2}\|\nb \om_1\|^2_{L_2(\R^3_{\eps})} +\nu \int_{\R}\om_1|^2_{r=\eps} dz \\ \le \frac{2}{\nu}\int_{\R^3_{\eps}}|u_1|^4 dx .  \eenn
Integrating with respect to time and passing with $\eps \rightarrow 0$ gives (\ref{2.12}). This concludes the proof.
\end{proof}
Introduce the quantities
\benn (\Phi, \Ga) = \left(\frac{\om_r}{r}, \frac{\om_{\va}}{r}\right) \eenn
which are solutions to the equations
\benn
\pa_t \Phi + v\cdot \nb \Phi - \nu (\De + \frac{2}{r})\Phi - (\om_r \pa_r + \om_z\pa_z)\frac{v_r}{r} =0 \\
\pa_t \Ga + v\cdot \nb \Ga - \nu (\De + \frac{2}{r})\Ga + 2 \frac{v_r}{r}\Phi =0
\eenn
\begin{remark}
In the proof of Theorem~1.1, Case 1 in \cite{CFZ} there is derived the formula (3.8) in \cite{CFZ} in the form
\ben \bal  \label{fi-ga} \|\Phi(t)\|^2_{L_2(\R^3)}+ \|\Ga(t)\|^2_{L_2(\R^3)}\\ +\nu \int_0^t (\|\tilde{\nb} \Phi(t)\|^2_{L_2(\R^3)} +\|\tilde{\nb}\Ga\|^2_{L_2(\R^3)}) dt' \\  \le \exp[c(1+\|r^dv_{\va}\|_{L_{p,q}(\R^3\times(0,t))})](\|\Phi(0)\|^2_{L_2(\R^3)}+ \|\Ga(0)\|^2_{L_2(\R^3)} )
\eal \een
where \benn \qquad \qquad \qquad  3/p+2/q \le 1-d, 0\le d< 1, \\ \frac{3}{1-d}  < p\le \infty, \frac{3}{1-d} \le q < \infty. \eenn
\end{remark}

Let us recall some properties of weak solutions to (\ref{NS}).
\begin{lemma}(see \cite{LSU}, Ch.2, Sect.3) \label{le-2.7}
For arbitrary $v\in L_{\infty}(0,T;L_2(\Om)$ $\cap L_2(0,T;H^1(\Om))$ the inequality holds:
\ben \label{2.19}
\qquad \qquad \|v\|^2_{L_{p,q}(\R^3\times (0,T))} \le c \sup_{0\le t \le T}\int_{\R^3} |v|^2 dx + c\int_0^T dt \int_{\R^3}|\nb v|^2 dx,
\een
where $$\frac{3}{p}+\frac{2}{q} \ge \frac{3}{2}.$$
\end{lemma}

\section{Sufficient conditions for regularity}
\setcounter{equation}{0}%
\setcounter{lemma}{0}%



Let \benn u_{\al} = \frac{u}{r^{\al}}, \quad \al \in (0,1). \eenn
Then $u_{\al}$ satisfies
\ben \label{3.1}
u_{\al,t} + v\cdot\nb u_{\al}+ \al\frac{v_r}{r}u_{\al}- \nu\De u_{\al}+\frac{2\nu(1-\al)}{r}u_{\al,r} + \frac{\nu\al(2-\al)}{r^2}u_{\al}=0, \nn \\ u_{\al}|_{t=0}=u_{\al}(0).
 \een
\begin{lemma}
Let $u(0) \in L_{\infty}(\R^3)$. Assume that $c_1$ is a constant and
\benn 
\int_0^T\int_{\R^3} \frac{v_r^2}{r^3} dx dt \le c_1^2,
\eenn
Let \benn c_2= \frac{1}{4\nu\al(2-\al)}\|u(0)\|^s_{L_{\infty}(\R^3)}c_1^2
+ \frac{1}{s} \|\|u(0)\|^s_{L_s(\R^3)} \\
{\rm where} \  \al=\frac{3}{s}, \ u_{\al}= v_{\va}r^{1-\al}.\eenn
Then
\ben \bal \label{3.3} r^d v_{\va}\in L_{p,q}(\R^3\times (0,T)), \ \frac{3}{p}+ \frac{2}{q} \ge \frac{3}{s}=\al,\\ d=1-\al \ {\rm and} \  \|r^d v_{\va}\|^s_{L_{p,q}(\R^3\times (0,T))} \le c_2. \eal \een
\end{lemma}
\begin{proof}
Multiplying (\ref{3.1}) by $u_{\al}|u_{\al}|^{s-2},$ integrating the result over $\R^3$ and
using that $u_{\al} \in C^{\infty}_0(\R^3\setminus I)$, we obtain
\benn
\qquad \frac{1}{s}\frac{d}{dt}\Or|u_{\al}|^s dx+ \al\Or\frac{v_r}{r}|u_{\al}|^s dx + \nu \Or|\nb |u_{\al}|^{s/2}|^2 dx \\ + \nu \al (2-\al)\Or\frac{|u_{\al}|^s}{r^2} dx = 0.
\eenn
The second term in the l.h.s. of the above equality can be estimated by
\benn
\left|\Or v_r|u_{\al}|^{s/2}\frac{|u_{\al}|^{s/2}}{r} dx \right| \le \frac{\eps}{2} \Or \frac{|u_{\al}|^{s}}{r^2} dx +\frac{1}{2\eps} \Or v_r^2|u_{\al}|^s dx \equiv I
\eenn
Using Lemma~\ref{le-2.4}, the second integral in $I$ is bounded by
\benn
\|u(0)\|^s_{L_{\infty}(\R^3)} \Or\frac{v_r^2}{r^{\al s}} dx.
\eenn
Employing this estimate, with $\eps=\nu\al(2-\al),$ integrating the result with respect to
time and using the density argument, yields
\ben \label{3.4} \bal
\frac{1}{s}\Or|u_{\al}(t)|^s dx + \nu \int_0^t dt' \Or|\nb |u_{\al}|^{s/2}|^2 dx \\ + \frac{\nu\al(2-\al)}{2} \int_0^t dt' \Or\frac{|u_{\al}|^s}{r^2} dx \\ \le \frac{\|u(0)\|^s_{L_{\infty}(\R^3)}}{4\nu\al(2-\al)}  \int_0^t dt' \Or\frac{v_r^2}{r^{\al s}}dx + \frac{1}{s}\Or |u_{\al}(0)|^s dx, \quad t\le T.
\eal \een
Let us assume that the r.hs. of (\ref{3.4}) is bounded by a constant $c_2.$ Then Lemma~\ref{le-2.7} implies
\benn
\|u_{\al}^{s/2}\|_{L_{p,q}(\R^3\times (0,T))} \le c_2^{1/2} \quad {\rm with} \quad \frac{3}{p}+ \frac{2}{q} \ge \frac{3}{2}.
\eenn
Hence
\benn
\|u_{\al}\|_{L_{p',q'}(\R^3\times (0,T))}^{s/2} =
\|u_{\al}\|^{s/2}_{L_{sp/2,sq/2}(\R^3\times (0,T))} \le c_2^{1/2}
\eenn where \benn \frac{3}{p'}+ \frac{2}{q'} \ge \frac{3}{2}\cdot \frac{2}{s}= \frac{3}{s}. \eenn
Comparing the above approach with (\ref{fi-ga}) we have that
$d=1-\al$ and the regularity criterion has the form
\benn
r^dv_{\va}\in L_{q'}(0,T;L_{p'}(\R^3)), \quad \frac{3}{p'}+\frac{2}{q'} \le 1-d,\quad 0\le d<1.
\eenn
Hence
\benn
\frac{3}{p'}+ \frac{2}{q'} \ge \frac{3}{s} = \al. \eenn
Therefore $\al s=3.$   \end{proof}

\begin{corollary}
From (\ref{3.3}), (\ref{fi-ga}) and (\ref{2.10}) we obtain
\benn 
\left\|\frac{v_{\va}^2}{r}\right\|_{L_2(\R^3)}^2 + \int_0^t \left\|\tilde{\nb}\frac{v_{\va}^2}{r}\right\|_{L_2(\R^3)}^2 dt' \le c_3, \quad t\le T, \eenn
where $c_3$ depends on the constants from the r.h.s. of (\ref{3.3}),(\ref{fi-ga}) and (\ref{2.10}).
\end{corollary}

 In view of (\ref{2.19}) we have
\ben \label{3.6}
\left\|\frac{v_{\va}^2}{r}\right\|_{L_{p,q}(\R^3\times (0,T))} \le cc_3 \een
where
\benn 
\frac{3}{p} + \frac{2}{q} \ge \frac{3}{2}.\eenn
Let $\R^3_{r_0} = \{x\in \R^3: r\le r_0 \}.$ Then (\ref{3.6}) implies
\ben \label{3.8}
 \left\|{v_{\va}^2}\right\|_{L_{p,q}(\R^3_{r_0}\times (0,T))}\le  \left\|\frac{v_{\va}^2}{r}\right\|_{L_{p,q}(\R^3_{r_0}\times (0,T))} \le cc_3
\een

Hence (\ref{3.8}) implies that $v_{\va} \in L_{p',q'}(\R^3_{r_0}\times (0,T))$ where
\benn 
\frac{3}{p'} + \frac{2}{q'} \ge \frac{3}{4}.\eenn

Consider Lemma~\ref{le-KP}. Let $s=4$. Then $r=8$, so \benn \frac{3}{4} + \frac{2}{8} =1 \eenn
Since $\frac{3}{4} < 1, v_{\va}$ satisfies assumptions of Lemma~\ref{le-KP}. Hence $v$ has no singular points in $\R^3_{r_0}\times (0,T).$

Next we show that axisymmetric solutions to problem (\ref{NS}) do not have singular points in the region located in a positive distance from the axis of symmetry.

\underline{}In \cite{CKN} is shown that singular points of $v$ in the axisymmetric case may appear on the axis of symmetry only. Therefore in any region located in a positive distance from the axis of symmetry there is no singular points of $v.$ However, we want to show that statement explicitly. Therefore, we proceed as follows.

Consider equation $(\ref{1.6-1.9})_1$. Let $\chi= \chi(r)$ be a smooth function such that
\benn
\chi(r) = \left\{ \begin{array}{lr} 1 & r\ge 2r_0 \\ 0 & r\le r_0    \end{array} \right.
 \eenn
We multiply $(\ref{1.6-1.9})_1$ by $\chi$ and introduce the notation $ \hat{v}_{\va} = v_{\va} \chi.$ Then $\hat{v}_{\va}$ satisfies
\ben \bal \label{3.10}
\hat{v}_{\varphi,t}+v\cdot\nabla \hat{v}_\varphi -\nu\Delta \hat{v}_\varphi +\frac{v_r}{r}\hat{v}_\varphi
+ \nu \frac{\hat{v}_\varphi}{r^2} \\ = v \cdot\nb \chi v_{\va}
-2\nu \nb v_{\va} \nb \chi - \nu v_{\va} \De \chi \\
\hat{v}_{\va}|_{t=0} = \hat{v}_{\va}(0).
 \eal \een
Multiplying (\ref{3.10}) by $\hat{v}_{\va}|\hat{v}_{\va}|^{s-2}$ and integrating over $\R^3\times (0,t)$ yields
\ben \bal \label{3.11} \frac{1}{s} \|\hat{v}_{\va}(t)\|^s_{L_s(\R^3)}  + \nu \int_0^t \|\nb|\hat{v}_{\va}|^{s/2}\|^2_{L_2(\R^3)} dt' \\ +
\nu \int_0^t \|\frac{|\hat{v}_{\va}|^{s/2}}{r}\|^2_{L_2(\R^3)} dt' + \int_0^t \Or \frac{v_r}{r} |\hat{v}_{\va}|^s dx dt' \\ = \int_0^t \Or (v\cdot \nb\chi v_{\va} - 2\nu \nb v_{\va} \nb\chi - \nu v_{\va} \De \chi)\hat{v}_{\va}|\hat{v}_{\va}|^{s-2} dx dt'.  \eal \een
In view of Lemma~\ref{le-2.7} the first two terms on the l.h.s. of (\ref{3.11}) are estimated from below by
\ben \label{3.12} \|\hat{v}_{\va}\|^s_{L_{\frac{5}{3}s}(\R^3\times(0,t))} \een
Next we estimate non-positive terms in (\ref{3.11}). In view of the energy estimate (\ref{2.1}) the last term on the l.h.s. of (\ref{3.11}) is bounded by \benn 
\|v_r\|_{L_{\frac{10}{3}}(\R^3\times(0,t))} \|\hat{v}_{\va}\|^s_{L_{\frac{10}{7}s}(\R^3\times(0,t))} \eenn
where the second factor can be always absorbed by (\ref{3.12}) because $5/3 >10/7.$
In view of (\ref{2.9}) and (\ref{2.1}) the first integral on the r.h.s. of (\ref{3.11}) is bounded by
\ben \label{3.14}
\|v\|_{L_{\frac{10}{3}}(\R^3\times(0,t))} \|v_{\va}\|_{L_4(\R^3\times(0,t))} \|\hat{v}_{\va}\|^{s-1}_{L_{\frac{20}{9}(s-1)}(\R^3\times(0,t))}. \een
The last factor in (\ref{3.14}) can be absorbed by (\ref{3.12}) for $\frac{20}{9}(s-1) \le \frac{5}{3}s$ which holds for $s\le 4.$ Then (\ref{3.12}) yields estimate for
\benn 
\|\hat{v}_{\va}\|_{L_{\frac{20}{3}}(\R^3\times(0,t))}.
 \eenn
Finally we estimate the last two integrals on the r.h.s. of (\ref{3.11}). We write them in the form
\benn
-\nu \int_0^t \Or (2\nb v_{\va} \nb\chi + v_{\va} \De\chi)\hat{v}_{\va}|\hat{v}_{\va}|^{s-2} dx dt'
\\ = -\nu \int_0^t \Or (2 \nb v_{\va} \nb\chi {v}_{\va}\chi|\hat{v}_{\va}|^{s-2}
+ v_{\va} \De\chi\hat{v}_{\va}|\hat{v}_{\va}|^{s-2}) dx dt'
\\  = -\nu \int_0^t \Or (\nb v^2_{\va} \nb\chi \chi |\hat{v}_{\va}|^{s-2}
+ v_{\va} \De\chi\hat{v}_{\va}|\hat{v}_{\va}|^{s-2}) dx dt' \\
= -\nu \int_0^t \Or (- v^2_{\va} \De\chi \chi |\hat{v}_{\va}|^{s-2}
-  v^2_{\va} |\nb\chi|^2 |\hat{v}_{\va}|^{s-2} \\ - v^2_{\va} \nb\chi \chi \nb|\hat{v}_{\va}|^{s-2}
+ v_{\va} \De \chi \hat{v}_{\va} |\hat{v}_{\va}|^{s-2})dx dt'
\\ = \nu \int_0^t \Or (v^2_{\va} |\nb\chi|^2 |\hat{v}_{\va}|^{s-2}
+ v^2_{\va} \nb\chi \chi \nb|\hat{v}_{\va}|^{s-2})dx dt'
= I_1+I_2,
\eenn
where
\benn |I_1| \le c\|v_{\va}\|_{L_4(\R^3)\times (0,t)}
\|\hat{v}_{\va}\|^{s-2}_{L_{2(s-2)}(\R^3)\times(0,t)}
\eenn
and the first factor is bounded in view of Remark~2.5 and the second factor
is absorbed by (\ref{3.12}) for $2(s-2)\le \frac{5}{2} s$ which holds for $s\le 12.$
Finally
\benn I_2= \nu(s-2)\int_0^t\Or v^2_{\va}\nb\chi\chi|\hat{v}_{\va}|^{s-3}\nb|\hat{v}_{\va}|dx dt'
\\ \le \frac{\nu(s-2)}{s} \int_0^t\Or v^2_{\va}\nb\chi\chi|\hat{v}_{\va}|^{s/2-2}\nb|\hat{v}_{\va}|^{s/2}dx dt'
\\ \le \frac{\eps}{2} \int_0^t\Or |\nb|\hat{v}_{\va}|^{s/2}|^2 dx dt'
+  \frac{c}{\eps} \int_0^t\Or |v_{\va}|^2|\hat{v}_{\va}|^{s-2} dx dt', \eenn
where the first integral is absorbed by the second term on the l.h.s. of (\ref{3.11})
and the second is bounded by
\benn \left(\int_0^t\Or |{v}_{\va}|^4 dx dt'\right)^{\frac{1}{2}}
\left(\int_0^t\Or |\hat{v}_{\va}|^{2(s-2)} dx dt'\right)^{\frac{1}{2}}=I_3\eenn

The first factor in $I_3$ is bounded in virtue of Remark~2.5 and the second factor is absorbed by
(\ref{3.12}) if $2(s-2)\le \frac{5}{3}s$ so $s\le 12.$
Summarizing, we obtain the estimate
\ben \label{3.16}
\|\hat{v}_{\va}\|_{L_{\frac{20}{3}}(\R^3\times(0,t))} \le c
\een
We observe that the estimate (\ref{3.16}) above is not strong enough to apply
Lemma~\ref{le-np2}(1) because  for $s=r$ it is required that $s\ge\frac{50}{7}.$

To increase regularity we introduce a new smooth cut-off function
\benn
\breve{\chi}(r) = \left\{ \begin{array}{lr} 1 & r\ge 3r_0 \\ 0 & r\le 2r_0    \end{array} \right.
 \eenn
Introducing notation $ \breve{v}_{\va}= {v}_{\va}\breve{\chi}$ we replace (\ref{3.10}) by
\benn 
\breve{v}_{\varphi,t}+v\cdot\nabla \breve{v}_\varphi -\nu\Delta \breve{v}_\varphi +\frac{v_r}{r}\breve{v}_\varphi
+ \nu \frac{\breve{v}_\varphi}{r^2} \\ = v \cdot\nb \breve{\chi} \hat{v}_{\va}
-2\nu \nb \hat{v}_{\va} \nb \breve{\chi} - \nu \hat{v}_{\va} \De \breve{\chi} \\
\breve{v}_{\va}|_{t=0} = \breve{v}_{\va}(0).
 \eenn
where we can use (\ref{3.16}). Hence increasing regularity of (\ref{3.16}) can be achieved
to meet assumptions of Lemma~\ref{le-np2}(1).
 This concludes the proof.
\begin{lemma} \label{le-3.3}
Assume that
\benn 
\int_0^T dt \Or \frac{\omega^2_{\va}}{r} dx < \infty
\eenn
Then there exists a constant $c$ such that
\ben \label{3.19}
\int_0^T dt \Or \frac{v_r^2}{r^3} dx \le c\int_0^T dt \Or \frac{\omega^2_{\va}}{r} dx
\een
\end{lemma}
\begin{proof}
From $(\ref{1.6-1.9})_3$ we have
\ben \label{3.20}
- \De\psi_{,z} + \frac{1}{r^2}\psi_{,z}= \omega_{\va,z},
\een
where
$v_r= -\psi_{,z}.$

To simplify further considerations we introduce the notation
\benn
u=\psi_{,z},\quad f= \om_{\va}
\eenn
Then (\ref{3.20}) takes the form
\ben \label{3.21} -\De u + \frac{1}{r^2} u= f_{,z}. \een
Recall that $I$ is the axis of symmetry. Let $C^{\infty}_0(\R^3\setminus I)$ be the set
of smooth functions vanishing near $I$ and outside a compact set.

Let $H^1_0(\R^3)$ be the closure of $C^{\infty}_0(\R^3\setminus I)$ in the norm
\benn \|u\|_{H^1_0(\R^3)}= \left(\Or (|\nb u|^2 + \frac{1}{r^2}|u|^2) dx\right)^{\frac{1}{2}}.
\eenn
Recall that functions from $H^1_0(\R^3)$ vanish on $I.$ By the weak solution to (\ref{3.21}) we
mean a function $u\in H^1_0(\R^3)$ satisfying the integral identity
\ben \label{3.22} \Or (\nb u\cdot \nb \chi + \frac{1}{r^2} u \chi)dx = - \Or f \chi_{,z} dx\een
which holds for any smooth function $\chi\in C^{\infty}_0(\R^3\setminus I)$ belonging to
$H^1_0(\R^3).$ Introducing the scalar product
\ben \label{3.23}
(u,v)_{H^1_0(\R^3)} = \Or (\nb u\cdot \nb v + \frac{1}{r^2} u v) dx
\een
we can write (\ref{3.22}) in the following short form
\benn 
 (u,\chi)_{H^1_0(\R^3)} = -(f, \chi_{,z})_{L_2(\R^3)} \eenn
where \benn (u,v)_{L_2(\R^3)} = \Or uv dx. \eenn
For $f\in L_2(\R^3),$ $ (f, \chi_{,z})_{L_2(\R^3)}$ is a linear functional on $H^1_0(\R^3).$
Hence we have
\benn |(f, \chi_{,z})_{L_2(\R^3)}| \le \|f\|_{L_2(\R^3)}\|\chi_{,z}\|_{L_2(\R^3)}
\le \|f\|_{L_2(\R^3)}\|\chi\|_{H^1_0(\R^3)}. \eenn
Hence by the Riesz Theorem there exists $F\in H^1_0(\R^3)$ such that
\benn -(f, \chi_{,z})_{L_2(\R^3)} = (F, \chi)_{H^1_0(\R^3)}. \eenn
Therefore there exists a solution to the integral identity (\ref{3.23}) such that $u=F\in H^1_0(\R^3)$
and the estimate holds
\benn 
\|u||_{H^1_0(\R^3)} \le c\|f\|_{L_2(\R^3)}. \eenn
It is clear that the solution is unique.

Since $u$ vanishes for $r=0,$ $u=\psi_{,z}$ so $\psi|_{r=0}=0$ also. Therefore we can
look for approximate weak solution satisfying the integral identity
\ben \label{3.26} \int_{\R^3_{\eps}} (\nb u \cdot\nb\chi + \frac{1}{r^2} u \chi) dx,
= - \int_{\R^3_{\eps}} f \chi_{,z} \een
where $\R^3_{\eps}= \{x\in \R^3: r>\eps \}, \eps>0.$ Let $\chi=\frac{\psi_{,z}}{r^{\al}}.$
Then recalling notation $u=\psi_{,z}, f=\om_{\va}$ identity (\ref{3.26}) takes the form
\benn  \int_{\R^3_{\eps}}(\nb \psi_{,z}\frac{\nb{\psi_{,z}}}{r^{\al}} +
\frac{1}{r^{2+\al}}|\psi_{,z}|^2)dx
= - \int_{\R^3_{\eps}} \frac{\om_{\va}}{r^{\al}}\psi_{,zz}. dx  \eenn
Performing differentiation we have
\ben \label{3.27} \bal
\int_{\R^3_{\eps}} \frac{1}{r^{\al}}|\nb\psi_{,z}|^2 dx -\al \int_{\R^3_{\eps}}\nb \psi_{,z}\psi_{,z} \nb r \ r^{-\al-1}dx
\\ + \int_{\R^3_{\eps}}\frac{\psi_{,z}^2}{r^{2+\al}}dx
 = -\int_{\R^3_{\eps}} \frac{\omega_{\va}}{r^{\al}}\psi_{,zz}dx.
\eal \een

The second integral on the l.h.s. of (\ref{3.27}) takes the form
\benn
-\al \int_{\R^3_{\eps}}\pa_r \psi_{,z}\psi_{,z} r^{-\al-1} r dr dz
= -\frac{\al}{2} \int_{\R^3_{\eps}}(\psi_{,z}^2)_{,r} r^{-\al} dr dz \\
= -\frac{\al}{2}\int_{\R^3_{\eps}} \pa_r (\psi_{,z}^2r^{-\al}) dr dz
- \frac{\al^2}{2}\int_{\R^3_{\eps}} \psi_{,z}^2 r^{-\al-1} dr dz
\\ = \frac{\al}{2}\int_{\R^3_{\eps}}\psi_{,z}^2r^{-\al}|_{r=\eps}dz
- \frac{\al^2}{2}\int_{\R^3_{\eps}} \psi_{,z}^2 r^{-\al-2} dx.
\eenn
Using this in (\ref{3.27}) and applying the H\"{o}lder and Young inequalities to the r.h.s.
term yields
\ben  \label{3.28} \int_{\R^3_{\eps}}\frac{1}{r^{\al}}|\nb\psi_{,z}|^2 dx
+\left(1-\frac{\al^2}{2}\right) \int_{\R^3_{\eps}}\frac{\psi_{,z}^2}{r^{2+\al}} dx
\le c \int_{\R^3_{\eps}} \frac{\omega_{\va}^2}{r^{\al}}dx.
\een
We have to emphasize that $\psi$ in (\ref{3.28}) is an approximate function. This should be denoted
with $\psi^{\eps}$ but we omitted it for simplicity.

Passing with $\eps \ra 0$, setting $\al=1$ and integrating this inequality with respect to time
implies (\ref{3.19}). This concludes the proof.
\end{proof}

Proof of Theorem~1
\begin{proof}
From (\ref{2.10}) and (\ref{2.11a}) we have
\benn  \left\|\frac{v_{\va}^2}{r}\right\|_{L_2(\R^3)}^2
+ \int_0^t \left\|{\nb}\frac{v_{\va}^2}{r}\right\|_{L_2(\R^3)}^2 dt' \\ \le
\exp(c \int_0^t \left\|\frac{\om_{\va}}{r}\right\|_{L_2(\R^3)}^{4/3}dt')\cdot
\left\|{\nb}\frac{v_{\va}^2(0)}{r}\right\|_{L_2(\R^3)}^2, \quad t\le T.
 \eenn
Next (\ref{fi-ga}) implies
\ben \label{23} \bal \left\|\frac{\om_r}{r}\right\|_{L_2(\R^3)}^2
+ \left\|\frac{\om_{\va}}{r}\right\|_{L_2(\R^3)}^2
\\ \le c \exp(\|u_{\al}\|_{L_{p,q}(\R^3\times(0,t))})
\left( \left\|\frac{\om_r(0)}{r}\right\|_{L_2(\R^3)}^2+
\left\|\frac{\om_{\va}(0)}{r}\right\|_{L_2(\R^3)}^2 \right) , \eal \een
\benn \frac{3}{p}+ \frac{2}{q} =\al,\ \  \al \le 1, \ t\le T.\eenn
Finally, in view of (\ref{1.10}) and Lemma~3.1 we get
\benn
 \|u_{\al}\|_{L_{p,q}(\R^3\times(0,t))} \le c(\|u(0)\|_{L_{\infty}(\R^3)}+ \|u(0)\|_{L_{s}(\R^3)}),
\eenn
where $\al=\frac{3}{s}, s\ge 3, t\le T.$
These estimates imply that
\benn
 \|v_{\al}\|_{L_{p',q'}(\R^3_{r_0}\times(0,t))}
\le h\left(\left\|\frac{v^2_{\va}(0)}{r}\right\|_{L_2(\R^3)},
\left\|\frac{\om_{\va}(0)}{r}\right\|_{L_2(\R^3)}, \right.
\\ \left. \left\|\frac{\om_{r}(0)}{r}\right\|_{L_2(\R^3)},   \|u(0)\|_{L_{\infty}(\R^3)},
\|u(0)\|_{L_{s}(\R^3)}\right), \quad \frac{3}{p'}+ \frac{2}{q'} \ge \frac{3}{4}
 \eenn
where $h$ is some positive increasing function of its arguments. Hence Lemma~2.3 implies local regularity.

To make statement more explicit we obtain from (\ref{23}) for $r< r_0$
and from Step 5 of the proof of Theorem~1 from \cite{NP2} that
\benn \|\om \|_{L_{\infty}(0,T;L_2(\R^3_{r_0}))} \le C(data)\eenn
\noindent where $data$ are data from the assumptions of the Step 1 of the proof of Theorem~1 from \cite{NP2}.

Considering the problem \benn \rot v = \om \\ \Div v =0 \eenn
and the local technique from \cite{LSU}, Ch.4, Sect.10, we have
\benn \|v \|_{L_{\infty}(0,T;L_6(\R^3_{r_0}))} \le C(data)\eenn
Consider the problem
\benn
v_t -\nu \De v + \nb p= -v \cdot \nb v \\
\Div v =0 \\
v|_{t=0} = v(0),
\eenn
Employing the result of Solonnikov, estimate for $v$ above and some interpolation we get
\benn \|v\|_{W^{2,1}_{\si, r}(\R^3_0 \times(0,T))} \le c(data)
+ c \|v\|_{B^{2-2/r}_{\si, r}(\R^3_0 \times(0,T))}\eenn
where $\si <6, r$ arbitrary. This proves the second part of Theorem~1.
\end{proof}

Proof of Theorem~2 follows from Theorem~1 and Lemma~\ref{le-3.3} .

\end{document}